\documentclass[a4paper]{amsart}
\usepackage{graphicx} 

\usepackage{amssymb}
\usepackage[leqno]{amsmath}
\usepackage{amsfonts}
\usepackage{amsopn}
\usepackage{amstext}
\usepackage{amsthm}

\usepackage{tikz}
\usetikzlibrary{cd}

\usepackage{enumerate}

\usepackage{verbatim}
\usepackage[colorlinks, backref]{hyperref}

\usepackage{mathrsfs} 
\usepackage{euscript} 

\usepackage{todonotes}
\usepackage[normalem]{ulem}
\usepackage{soul}
\usepackage{cancel}
\newtheorem{theorem}{Theorem}[section]
\newtheorem{corollary}[theorem]{Corollary}
\newtheorem{proposition}[theorem]{Proposition}
\newtheorem{lemma}[theorem]{Lemma}
\theoremstyle{definition}
\newtheorem{definition}[theorem]{Definition}
\newtheorem{remark}[theorem]{Remark}
\newtheorem{example}[theorem]{Example}

\newtheorem{thmx}{Theorem}

\newcommand \Homeo{ \mbox{Homeo}}

\newcommand{\bP}{{\mathbb{P}}}
\newcommand{\QQ}{{\mathbb{Q}}}
\newcommand{\NN}{{\mathbb{N}}}
\newcommand{\RR}{{\mathbb{R}}}
\newcommand{\ZZZ}{\mathbb{Z}}

\newcommand{\Cc}{{\mathcal{C}}}

\newcommand{\Kk}{{\mathcal{K}}}
\newcommand{\Mm}{{\mathcal{M}}}
\newcommand{\Pp}{{\mathcal{P}}}

\newcommand{\FF}{\mathcal{F}}
\newcommand{\PP}{\mathcal{P}}
\newcommand{\CC}{\mathcal{C}}

\newcommand \fra{Fra\"{i}ss\'{e} }
\newcommand{\Can}{2^\NN}

\newenvironment{exm}[1]{\begin{example}\label{exm:#1}}{\end{example}}

\newenvironment{pro}[1]{\begin{proposition}\label{pro:#1}}{\end{proposition}}

\newcommand{\PRO}[1]{Proposition~\textup{\ref{pro:#1}}}

\newcommand{\ffF}{\mathscr{F}}
\newcommand{\ggG}{\mathscr{G}}

\newcommand{\dd}{\colon}
\newcommand{\OPN}[1]{\operatorname{#1}}

\newcommand{\TFCAE}{The following conditions are equivalent:}

\newcommand{\Cantor}{Cantor space}

\newcommand{\dint}[1]{\,\textup{d} #1}

\begin{document}

\title[Automorphism groups of measures, I]{Automorphism groups of measures\\ on the Cantor space.\\Part I: Good measures and Rokhlin properties}

\begin{abstract}

We study criteria for the existence of a dense or comeager conjugacy class in the automorphism group of a given measure on the Cantor space. 
We concentrate on good measures, defined by Akin [\emph{Trans.\ Amer.\ Math.\ Soc.} \textbf{357} (2005), no. 7, 2681--2722], 
which we characterize as a particular subclass of ultrahomogeneous measures. We 
determine 
good measures with rational values on clopen sets whose  automorphism group admits a comeager conjugacy class. Our approach uses the \fra theory.
\end{abstract}

\author[M.~Doucha]{Michal Doucha}
\address[M. Doucha]{Institute of Mathematics\\
Czech Academy of Sciences\\
\v Zitn\'a 25\\
115 67 Praha 1\\
Czech Republic}
\email{doucha@math.cas.cz}
\author[D.~Kwietniak]{Dominik Kwietniak}
\address[D.~Kwietniak]{Faculty of Mathematics and Computer Science\\
Jagiellonian University in Krak\'{o}w\\
\L{}o\-ja\-sie\-wi\-cza 6\\%
 30-348 Krak\'{o}w\\Poland}
\email{dominik.kwietniak@uj.edu.pl}
\author[M.~Malicki]{Maciej Malicki}
\address[M. Malicki]{Faculty of Mathematics, Mechanics and Computer Science of the Warsaw University}
\email{mamalicki@gmail.com}
\author[P.~Niemiec]{Piotr Niemiec}
\address[P.~Niemiec]{Wydzia\l{} Matematyki i~Informatyki\\%
 Uniwersytet Jagiello\'{n}ski\\ul.\ \L{}o\-ja\-sie\-wi\-cza 6\\%
 30-348 Krak\'{o}w\\Poland}
\email{piotr.niemiec@uj.edu.pl}
\thanks{Research partially supported by the National Center of Science, Poland
 under the Weave-UNISONO call in the Weave programme [grant no
 2021/03/Y/ST1/00072]. The first-named author was supported by the GA\v{C}R project 22-07833K and by the Czech Academy of Sciences (RVO 67985840).}

\maketitle

\section{Introduction}
We study automorphism groups of measures of Cantor spaces. 
The Cantor space $\Can$ is homeomorphically unique, but it carries multiple non-homeo\-morphic measures (here, all measures are always Borel and probabilistic).  The automorphism group of a measure $\mu$ is the collection $\Homeo(\Can,\mu)$ of all homeomorphisms of $\Can$ that preserve $\mu$. 
It is a closed subgroup of $\Homeo(\Can)$.

A topological invariant for a full (positive on nonempty open sets), nonatomic measure $\mu$ on the Cantor space is its clopen values set $V_\mu$ (the set of values attained on clopen subsets). Good measures on $\Can$ are full measures $\mu$ satisfying the Subset Condition: if $\mu(U) < \mu(V)$ for some clopen sets $U$ and $V$, then $V$ contains a clopen subset $W$ with $\mu(W) = \mu(U)$. All good measures are nonatomic (see Proposition \ref{PN-good-nonatomic}). Good measure $\mu$ is up to a homeomorphism determined by its clopen values set $V_\mu$. Glasner and Weiss \cite{GW95} proved that if $\mu$ is a unique invariant measure for some minimal homeomorphism on $\Can$ then it is good. Akin \cite{Akin05} proved the converse, that is, he proved that every good measure is preserved by a uniquely ergodic homeomorphism of $\Can$. He \cite{Akin05} also proved that  $\Homeo(\Can,\mu)$ contains a comeager
conjugacy class when $V_\mu+\ZZZ$ is a $\QQ$-vector subspace of $\RR$, a criterion later recovered by Ibarluc\' ia and Melleray \cite{IM16}.

Ibarluc\' ia and Melleray \cite{IM16} studied existence of dense conjugacy class in the group $\Homeo(\Can,\mu)$ of all homeomorphisms preserving a given good measure $\mu$. They used a \fra theoretic approach to recover a result of Akin which provides a criterion phrased in terms of $V_\mu$ and describing when the group $\Homeo(\Can,\mu)$ admits a comeager conjugacy class. See also \cite{IM17,Mel23} for further research on the topic.

We apply the \fra theoretic approach to good measure themselves, an idea used in particular cases already in \cite{KeRo07} and also in \cite{BieKuWa19}, to obtain the following. We refer to Section~\ref{sec:good-ultra} for undefined notions.
\begin{thmx}[see Theorem~\ref{thm:fraisseisgood}]
Good measures on the Cantor space are precisely those full ultrahomogeneous measures that are maximal.   
\end{thmx}



As ultrahomogeneity is
closely related to \fra classes, it is quite natural to study
the automorphism groups of all ultrahomogeneous measures in terms of properties
typical in dynamical systems (such as transitivity or minimality of the action,
Rokhlin properties, etc.). These are the topic of the second part of the paper
\cite{DKMN2}, whereas here in this article we will focus mainly on good measures (as they are the most common examples from the aforementioned class).
Among our results on the automorphism groups of good measures we highlight the following.
\begin{thmx}[see Theorem~\ref{thm:Q-ring}]
Let $\mu$ be a good measure on $\Can$ whose clopen values set is contained in \(\QQ\). \TFCAE
\begin{itemize}
    \item $\Homeo(\Can,\mu)$ admits a dense conjugacy class.
    \item $\Homeo(\Can,\mu)$ admits a comeager conjugacy class.
    \item $V_\mu+\ZZZ$ is a subring of \(\RR\).
\end{itemize}  
\end{thmx}


The paper is organized as follows. In Section~\ref{sect:prelim} we recall
the notion of a group-like sets and introduce general concepts from category
theory that will be used by us in the sequel. Section~\ref{sec:good-ultra} is devoted
to the concept of ultrahomogeneity for measures. We give there two
characterizations of good measures: one in terms of \fra limits (of certain
countable structures) and the other as precisely those ultrahomogeneous measures
that are maximal (see, respectively, Proposition~\ref{prop:fraissemeasure} and
Theorem~\ref{thm:fraisseisgood} below). We also give illustrative examples of
ultrahomogeneous measures that are not good, and show that good measures are (in
a certain quite natural sense) generic among all measures whose clopen values
set is contained in a given group-like set. Section~\ref{sec:amalgam-1} discusses Rokhlin 
properties of the automorphism groups of good measures. Among other things, we
give there an equivalent condition for the aforementioned group to have
the Rokhlin property (see Theorem~\ref{th:Dense} below) as well as a sufficient
condition for this group to have a non-meager conjugacy class (consult
Theorem~\ref{th:CAMImpliesNonMeager} below). 
The latter is also an equivalent condition but we omit the proof of the other direction as we do not need it.
Finally, in Section~\ref{sec:amalgam-2} 
we apply the results to fully classify (in Theorem~\ref{thm:Q-ring}) good measures with clopen values
set contained in \(\QQ\) whose automorphism groups have the Rokhlin property and show that for such measures this property is actually equivalent to the strong Rokhlin property. The same equivalence (between strong and standard Rokhlin properties) holds also for good measures whose clopen values set contains \(\QQ \cap [0,1]\) (as shown by
Lemma~\ref{lem:Q-in-V}). 

\subsection*{Notation and terminology}
In this paper \(\NN = \{1,2,\ldots\}\), all topological spaces are Polish (that is, they are completely metrisable and separable), apart from the space \(\OPN{Full}(V)\) introduced in the paragraph preceding \PRO{PN-generic-measure}. All measures considered here are Borel probability measures. A Cantor space is a zero-dimensional, perfect,
nonempty, compact Polish space. We write $\Can$ to denote the countable
topological product of the discrete spaces $\{0,1\}$. It is standard that $\Can$ is a Cantor space. A measure is \emph{full} if it is positive on all nonempty
open sets (has full support). A measure is \emph{nonatomic} if it vanishes on all
countable subsets. A homeomorphism $h\colon X \to X$ is \emph{minimal} if every orbit is dense. For any measure \(\mu\) on \(X\) we use \(h^*(\mu)\) to denote
the \emph{transport} of \(\mu\) via \(h\); that is, \(h^*(\mu)(B) =
\mu(h^{-1}(B))\) for any Borel set \(B \subseteq X\). The measure $\mu$ is \emph{invariant} under \(h\) (or, equivalently, \(h\) \emph{preserves} \(\mu\))
if $h^*(\mu) = \mu$.
Given a measure $\mu$ on $\Can$ we write $\Homeo(\Can,\mu)$ for the Polish group of all homeomorphisms of $\Can$ preserving $\mu$.\par
Following e.g. \cite{Akin05}, we say a Polish group \(G\) has \emph{Rokhlin property} (resp. \emph{strong Rokhlin property}) if it admits a dense (resp. dense \(\ggG_{\delta}\)) conjugacy class.

\section{Preliminaries}\label{sect:prelim}
\subsection{Group-like sets.}
Recall that a set $V$ is \emph{group-like}
if $\{0,1\}\subseteq V\subseteq[0,1]$ and $V+\mathbb{Z}$ is an additive subgroup of $\mathbb{R}$. In other words, 
$V\subseteq[0,1]$ is group-like if and only if $\{0,1\}\subseteq V$ and
there is an additive subgroup $G$ of $\mathbb{R}$ such that $V=[0,1]\cap G$. \textbf{In this paper, all group-like sets are, by definition, (in addition) countably infinite}. We say that a group-like set $V$ is \emph{$\QQ$-like} (respectively, \emph{ring-like}) if $V + \mathbb{Z}$ is a rational vector space (respectively, a subring of $\mathbb{R}$). It is easy to see that every group-like set must be dense in $[0,1]$.

\begin{lemma}
\label{le:CommonPartition}
Let $V$ be group-like. Let $x^0_i \in V$ ($1\leq i \leq k$), $x^1_j \in V$ ($1\leq j \leq l$), and $z \in V$ be such that $$x^0_1+ \ldots +x^0_k= x^1_1+ \ldots +x^1_l=z.$$ Then there exist $m\ge 1$ and $z_i \in V$ ($1\leq i \leq m$) such that $$z_1+ \ldots + z_m=z,$$ and $\{1, \ldots m\}$ 
 can be partitioned in two ways: into the sets $X^0_i$ ($1\leq i \leq k$) and  into sets $X^1_j$ ($1\leq j \leq l$) in such a way that for $1\leq i \leq k$ and $1\leq j \leq l$ we have $$x^0_i= \sum_{s \in X^0_i} z_s \quad\text{and}\quad x^1_j= \sum_{s \in X^1_j} z_s.$$ 
\end{lemma}

\begin{proof}
The proof goes by induction on $n=k+l$. If $k=1$ or $l=1$, then the lemma is obviously true. Assume that $k,l \geq 2$ and that the lemma holds whenever $k'+l'< k+l$. Without loss of generality we can assume $x^0_k \geq x^1_l$. If $x^0_k= x^1_l$, we apply the inductive assumption to $x^0_1, \ldots,x^0_{k-1}$, $x^1_1, \ldots,x^1_{l-1}$ and $\bar{z}=z-x^1_l$ to obtain $m$ and $z_s \in V$ $(1\leq s \leq m)$. Clearly, $z_1, \ldots ,z_m, z_{m+1}=x^1_l$ are as required. Otherwise,  we slightly modify the first sequence setting $\bar{x}^0_1=x^0_1,\ldots, \bar{x}^0_{k-1}=x^0_{k-1}$, $\bar{x}^0_k=(x^0_k-x^1_l)$, and then apply the inductive assumption to $\bar{x}^0_1,\ldots, \bar{x}^0_k$, $x^1_1, \ldots,x^1_{l-1}$, and $\bar{z}=z-x^1_l$ to obtain $m$ and $z_s \in V$ $(1\leq s \leq m)$. Again, $z_1, \ldots ,z_m, z_{m+1}=x^1_l$ are as required. 
\end{proof}

\subsection{Categories}
Let $\Kk$ be a category. Writing $A\in\Kk$ we mean that ``$A$ is an object of $\Kk$.'' We say that $\Kk'$ is a subcategory of $\Kk$ and write $\Kk'\subseteq \Kk$ if $\Kk'$ is a category such that each
object of $\Kk'$ is an object of $\Kk$ and each morphism arrow of $\Kk'$ is also an arrow of $\Kk$ (with the same domain and co-domain).

For $A, B \in \Kk$ and a morphism $p\colon B \rightarrow A$ that is a surjective mapping, we say that \emph{$p$ projects $B$ onto $A$}, and that $B$ is a \emph{lift} (or, more precisely, a \emph{$p$-lift}) of $A$. For $A \in \Kk$, we write $\Kk^A \subseteq \Kk$ for the subcategory of all lifts of $A$.

We say that $\Kk$ is \emph{directed} if every two objects have a common lift, that is, for any $A_i \in \Kk$ ($i=0,1$), there is $B\in\Kk$ and there are morphisms $r_i\colon B \rightarrow A_i$. The category $\Kk$ \emph{has amalgamation} if for any $A \in \Kk$ and any morphisms $r_i\colon B_i \rightarrow A$ ($i=0,1$), there is $C\in\Kk$ and there are morphisms $s_i\colon C \rightarrow B_i$ such that $r_0 \circ s_0=r_1 \circ s_1$. A subcategory $\Kk'\subseteq\Kk$ is \emph{cofinal} in $\Kk$ if every object from $\Kk$ has a lift from \(\Kk'\) (with an arrow from \(\Kk\)). \(\Kk\) has \emph{cofinal amalgamation} if there is a cofinal $\Kk' \subseteq \Kk$ with amalgamation. Finally, the category $\Kk$ has \emph{weak amalgamation} if for any $A \in \Kk$ there is a morphism $p\colon A' \rightarrow A$ such that for any morphisms $r_i\colon B_i \rightarrow A'$ ($i=0,1$), there is $C\in\Mm$ and there are morphisms $s_i\colon C \rightarrow B_i$ such that $p \circ r_0 \circ s_0= p \circ r_1 \circ s_1$. We call any $A'$ as above an \emph{amalgamation basis for $A$}.

We denote by $\sigma \Kk$ the space of all \emph{chains} from $\Kk$, i.e., sequences $(A_n,p_n)$ such that $A_n \in \Kk$ and $p_n\colon A_{n+1} \rightarrow A_n$ are morphisms. We will denote by $p^j_i$, $i<j$, the appropriate compositions of $p_n$, and write $(A_n)$ for $(A_n,p_n)$ if there is no risk of ambiguity. (Iso)morphisms of chains are defined in a natural way (see \cite{Kub22} for details).
We say that a chain $(A_n,p_n)$ from $\Kk$ is \emph{\fra }if $\{A_n\}$ is cofinal in $\Kk$, and for every $i$, $A \in \Kk$, and $\pi\colon A \rightarrow A_i$ there are $j \geq i$ and $r\colon A_j \rightarrow A$ such that $\pi \circ r=p^j_i$. Any two \fra chains are isomorphic. It is the cornerstone of \fra theory that if the category $\Kk$ is countable, it is directed and has amalgamation then $\sigma\Kk$ admits a \fra chain. We refer to \cite{Kub22} for details.


\section{Good measures are ultrahomogeneous}\label{sec:good-ultra}

In the existing literature good measures are \emph{always} defined as full,
nonatomic and satisfying the subset condition. To our best
knowledge, the observation that being nonatomic is redundant in this definition is new.

\begin{proposition}\label{PN-good-nonatomic}
A full measure on $\Can$ satisfying the subset condition is good.
\end{proposition}
\begin{proof}
Let \(\mu\) be a full measure on $\Can$ satisfying the subset condition. It remains to show that \(\mu\) is nonatomic. Assume that
\(\mu(\{x\}) = c > 0\) for some \(x \in \Can\). Since there exists $y\in\Can$ such that \(\mu(\{y\})=0\), there are clopen sets
\(U,V\) with $x\in U$ and $y\in V$ such that \(0<\mu(V) < c\) and \(\mu(U) < c + \mu(V)\). Hence, \(\mu(U) \geq c > \mu(V)\).  The subset condition yields
a clopen set \(W\subseteq U\) with  \(\mu(W) = \mu(V)\). We infer
that \(x \notin W\) (as \(\mu(W) < c\)). Hence \(\mu(W) \leq \mu(U \setminus
\{x\}) = \mu(U) - c < \mu(V)\), which contradicts our choice of \(W\),
and finishes the proof.
\end{proof}

Let $\Pp$ be the category of all finite partitions of the Cantor space $\Can$ into clopen sets (henceforth we refer to these as \emph{clopen partitions}), with surjections given by the refinement relations as morphisms.  Observe that $\Pp$ has amalgamation. In the sequel, the letter $\pi$ will always refer to morphisms in $\Pp$, e.g.,  $\pi^V_U\colon V \rightarrow U$. If there is no risk of confusion, we will write $\pi$ for $\pi^V_U$. In particular, a \emph{$\pi$-lift} is a lift by some refining partition.

Let $V$ be a group-like set. We define $\FF_V$ to be the category of pairs $(P,\mu)$, where $P\in \Pp$ and $\mu$ is a probability distribution on $P$ such that for every $x\in P$ we have $\mu(x)\in V$ and $\mu(x)>0$. 

In $\FF_V$, an \emph{$\FF_V$-morphism} between $(P_\nu,\nu)$ and $(P_\mu,\mu)$ is a surjection $p\colon P_\nu \to P_\mu$ in $\Pp$ such that for every $x\in P_\mu$ we have \(\mu(x) = \sum_{y \in p^{-1}(\{x\})} \nu(y)\). Clearly, every chain $((P_{\mu_n},\mu_n),\pi_n)$ in $\FF_V$ (i.e., $\pi_n$ are morphisms in $\Pp$ that are also morphisms in $\FF_V$) determines a full measure $\mu$ such that $\mu$ is defined on the inverse limit of $(P_{\mu_n},\pi_n)$ (which is a zero-dimensional compact metrisable space) and the set of clopen values of $\mu$ (the set of values that $\mu$
 attains on clopen sets) is a subset of $V$. Conversely, any full measure on $\Can$ (or more generally, on a zero-dimensional compact metrisable space) can be represented as such a chain. For a full measure $\mu$ on $\Can$, and $P,R \in \Pp$, we say that a surjection $p\colon R \to P$ is a \emph{$\mu$-morphism} if it is an $\FF_V$-morphism for restrictions of $\mu$ to $P$ and $R$. Note that for every $\mu$, every morphism $\pi$ in $\Pp$ determines $\mu$-morphism, but the converse is false.

A measure $\mu$ on $\Can$ satisfies subset condition if and only if it is $V_\mu$-homogeneous as defined in \cite{BieKuWa19}.

\begin{definition}\label{dfn:ultrahomo}
Let $\mu$ be a full measure on $\Can$ and let $V=V_\mu$ be the clopen values set of $\mu$. We say that $\mu$ is \emph{ultrahomogeneous} if for any $P, R \in \Pp$ such that there is a bijection $f\colon P \to R$ satisfying $\mu(x)=\mu(f(x))$ for $x\in P$ (for simplicity, we will call each such a function \(f\) a \emph{partial \(\mu\)-isomorphism}), there exists a homeomorphism $\varphi\in\Homeo(\Can,\mu)$ such that for every $x\in P$ it holds $\varphi[x]=f(x)$ where \(\varphi[x]\) denotes the image of \(x\) via \(\varphi\). 

We say that a measure $\mu$ on the Cantor space $\Can$ is \emph{maximal} if for every tuple $v_1,\ldots,v_n$ of non-zero elements of $V(\mu)$ such that $\sum_{i=1}^n v_i=1$ there exists a clopen partition $\{x_1,\ldots,x_n\}$ of $\Can$ such that $\mu(x_i)=v_i$ ($1\leq i\leq n$), and moreover for every $v,w\in V(\mu)$ we have $|v-w|\in V(\mu)$.
\end{definition}

\begin{proposition}\label{prop:fraissemeasure}
Let $V$ be a group-like set. Then $\FF_V$ is countable, directed, and has amalgamation. In particular, there exists a maximal and ultrahomogeneous measure on $\Can$ with clopen values set equal to $V$.
\end{proposition}

\begin{proof}Clearly, $\FF_V$ is countable. Since directedness follows in this case from amalgamation, we only prove the latter.

Fix $(E_1,\mu_1)$, $(E_2,\mu_2)$, $(F,\nu)$ in $\FF_V$ and consider morphisms $f_i\colon (E_i,\mu_i)\to (F,\nu)$, $(i=1,2)$. Enumerate $F$ as $\{x_1,\ldots,x_l\}$. Fix $1\leq t\leq l$. Let \[f_1^{-1}(x_t)=\{y^t_1,\ldots,y^t_{n(t)}\}\subseteq E_1\quad\text{and}\quad f_2^{-1}(x_t)=\{z^t_1,\ldots,z^t_{m(t)}\}\subseteq E_2.\] By the assumption, we have \[\nu(x_t)=\sum_{i=1}^{n(t)} \mu_1(y^t_i)=\sum_{j=1}^{m(t)} \mu_2(z^t_j).\] Since $V$ is group-like, by Lemma~\ref{le:CommonPartition} there are $k(t)$ and $w^t_1,\ldots,w^t_{k(t)}\in V$ such that $\sum_{i=1}^{k(t)} w^t_i=\nu(x_t)$ and there are partitions $\{Y^t_i:1\le i\leq n(t)\}$ and $\{Z^t_j:1\le j\leq m(t)\}$ of $\{1,\ldots,k(t)\}$ so that for $1\le i\leq n(t)$ and $1\le j\leq m(t)$ it holds\[\mu_1(y^t_i)=\sum_{l\in Y^t_i} w^t_l,\quad\text{and}\quad\mu_2(z^t_j)=\sum_{l\in Z^t_j} w^t_l.\]

We take any clopen partition $G$ of $\Can$ with $k(1)+\ldots+k(l)$ elements. Then we enumerate elements of $G$ so that $G$ can be written as a disjoint union
\[G=\bigcup_{1\le t\leq l}\{g^t_1,\ldots,g^t_{k(t)}\}.\]
Finally, we set $(G,\mu_G)\in\FF_V$ so that  $\mu_G(g^t_i)=w^t_i$, for $1\le t\leq l$ and $1\le i\leq k(t)$. We define maps $p_i\colon (G,\mu_G)\to (E_i,\mu_i)$  $(i=1,2)$, by setting \[p_1(w^t_l)=y_i,\text{ where $1\le i\leq n(t)$ is such that }l\in Y^t_i\]
and
\[p_2(w^t_l)=z_j,\text{ where $1\le j\leq m(t)$ is such that }l\in Z^t_j.\]
We leave to the reader the straightforward verification that $p_1$ and $p_2$ are well-defined $\FF_V$-morphisms and $f_1\circ p_1=f_2\circ p_2$, thus the amalgamation is verified.

Let $((P_n,\mu_n),\pi_n)$ be a \fra chain in $\sigma \FF_V$. We claim that the inverse limit of $((P_n,\mu_n),\pi_n)$ is a measure on $\Can$ that is full and moreover maximal and ultrahomogeneous. It is straightforward that the inverse limit of $P_n$'s is infinite and has no isolated points, thus we identify it with $\Can$. Moreover, the inverse limit $\mu$ of $\mu_n$'s defines a measure on $\Can$. By the definition of $\FF_V$, $\mu$ is positive on every nonempty clopen set, therefore $\mu$ is full. Since $V$ is group-like, to check maximality, it suffices to check that for every tuple $v_1,\ldots,v_n$ of non-zero elements of $V$ such that $\sum_{i=1}^n v_i=1$ there exists a clopen partition $\{x_1,\ldots,x_n\}$ of $\Can$ such that $\mu(x_i)=v_i$ ($1\leq i\leq n$). To see this, take $(P_\nu,\nu)\in\FF_V$, such that $P_\nu=\{z_1,\ldots,z_n\}$ and $\nu(z_i)=v_i$ $(1\le i\leq n)$. Since $((P_n,\mu_n),\pi_n)$ is a \fra chain in $\sigma \FF_V$ we find $j$ such that there is a morphism $p\colon (P_{\mu_j},\mu_j)\to (P_\nu,\nu)$. Now, set $x_i=p^{-1}(z_i)$ ($1\leq i\leq n$). 

Finally, the fact that $\mu$ is ultrahomogeneous immediately follows from the fact that the \fra chain $((P_n,\mu_n),\pi_n)$ is homogeneous in the sense of \cite[Section 5]{Kub22}.
\end{proof}
\begin{theorem}\label{thm:fraisseisgood}
Let $\mu$ be a full measure on the Cantor space. Then $\mu$ is maximal and ultrahomogeneous if and only if it is good.
\end{theorem}
\begin{proof}
Suppose first that $\mu$ is maximal and ultrahomogeneous. We need to check that $\mu$ satisfies the subset condition. Let $V\subseteq [0,1]$ be the clopen values set of $\mu$. Let $A$ and $B$ be clopen subsets of $\Can$ such that $\mu(A)<\mu(B)$. We may suppose that $\mu(B)<1$, otherwise $B=\Can$ and there is nothing to prove. Since $\mu$ is maximal there exists a clopen partition $\{P_1,P_2,P_3\}$ of $\Can$ such that $\mu(P_1)=\mu(A)$, $\mu(P_2)=\mu(B)-\mu(A)$, and $\mu(P_3)=1-\mu(B)$. Define $f\colon \{P_1\cup P_2,P_3\}\to \{B,\Can\setminus B\}$ by $f(P_1\cup P_2)=B$ and $f(P_3)=\Can\setminus B$. Since $f$ is a morphism and $\mu$ is ultrahomogeneous there exists a homeomorphism $\varphi\in\Homeo(\Can,\mu)$ such that $\varphi[P_1\cup P_2]=B$. Set $A':=\phi[P_1]$. Then $A'\subseteq B$ and $\mu(A')=\mu(A)$. This shows that $\mu$ satisfies the subset condition. Since $\mu$ is a full measure by assumption, $\mu$ must be nonatomic, hence it is good.

Conversely, suppose that $\mu$ is good. Then by \cite[Proposition 2.4]{Akin05} the clopen values set $V$ of $\mu$ is group-like. So by Proposition~\ref{prop:fraissemeasure} there exists a maximal and ultrahomogeneous measure $\nu$ with clopen values set $V$. By the first part of this proof, $\nu$ is good. Since $\mu$ and $\nu$ are two good measures with the same clopen values set they must be conjugate (that is, there is $\varphi\in\Homeo(\Can)$ such that $\varphi^*(\mu)=\nu$) by \cite[Theorem 2.9 (a)]{Akin05} (see also, \cite[Theorem 2.3 (d)]{AkDoMauYi08}). It is plain that a measure homeomorphic to a maximal and ultrahomogeneous measure is again maximal and ultrahomogeneous, and so is $\mu$.
\end{proof}


In view of Theorem~\ref{thm:fraisseisgood}, ultrahomogeneous measures are a priori a weaker notion than good measures if they are not assumed to be maximal. We confirm this by providing examples of ultrahomogeneous measures that are not maximal, therefore they are not good. It may be interesting to study such measures and homeomorphisms preserving them in a similar fashion as for good measures. In a subsequent paper (that is, in the second part of the present article), we will propose a general approach to this issue. Among other things, we will give there an example of a full measure \(\mu\) on the Cantor space such that \(\mu\) is ultrahomogeneous and not good. Moreover, the action of \(G = \Homeo(\Can,\mu)\) on $\Can$ is minimal and \(G\) has the Rokhlin property.

\begin{exm}{non-good-ultra}
Let $V_1$ be the set of $3$-adic rationals inside $[0,1]$. Moreover, set $V_2:=\{0,1\}\cup \{n\cdot \alpha\; (\text{mod }1)\colon n\in\mathbb{Z}\}$, where $\alpha\in\mathbb{R}$ is some irrational number. Clearly, both $V_1$ and $V_2$ are group-like sets, and therefore there exist corresponding good measures $\mu'_i$ with clopen values set being equal to $V_i$, for $i=1,2$.

Write now the Cantor space as a disjoint union of two nonempty clopen sets $C_1$ and $C_2$. Let $\mu_1$ be the measure supported on $C_1$ that is equivalent to $1/3 \mu'_1$ via some homeomorphism between $\Can$ and $C_1$. Similarly, let $\mu_2$ be the measure supported on $C_2$ that is equivalent to $2/3 \mu'_2$ via some homeomorphism between $\Can$ and $C_2$. Finally, set $\mu:=\mu_1+\mu_2$. Clearly, $\mu$ is a full measure on the Cantor space. We claim it is ultrahomogeneous, but not maximal. For ultrahomogeneity, take two clopen partitions $\PP_1$ and $\PP_2$ such that there is a bijection $f:\PP_1\to\PP_2$ satisfying $\mu(P)=\mu(f(P))$, for every $P\in\PP_1$. Without loss of generality,
we may assume that every $P\in\PP_i$, for $i=1,2$, is either a subset of $C_1$ or a subset of $C_2$. Indeed, if $A_i,A'_i\subseteq C_i$, $i=1,2$, are clopen sets such that $\mu(A_1)+\mu(A_2)=\mu(A'_1)+\mu(A'_2)$ then $\mu(A_i)=\mu(A'_i)$ for $i=1,2$, since otherwise $\mu(A_1)-\mu(A'_1)$ would be an irrational number, a contradiction.

If $P\in\PP_1$ is a subset of $C_1$ then $\mu(P)$ is rational, therefore also $f(P)\subseteq C_1$. Analogously, if $P\in\PP_2$ is a subset of $C_2$ then also $f(P)\subseteq C_2$. Since $\mu'_i$, $i=1,2$, are ultrahomogeneous there exist measure preserving homeomorphisms $\phi_i\colon C_i\to C_i$, $i=1,2$, such that for every $P\in\PP_1$ a subset of $C_1$, resp. a subset of $C_2$, we have $\phi_1(P)=f(P)$, resp. $\phi_2(P)=f(P)$. These two homeomorphisms together form the desired measure preserving homeomorphism on $\Can$.

Finally, we show that $\mu$ is not maximal, therefore not good. The clopen values set $V$ of $\mu$ is $1/3 V_1+2/3 V_2$. Therefore $1/3\in V$. If $\mu$ were maximal there would exist a clopen partition $\{P_1,P_2,P_3\}$ such that $\mu(P_i)=1/3$, for $i=1,2,3$. It is clear that by the definition of $\mu$ such a partition does not exist. Note also that \(\Homeo(\Can,\mu)\) is quite rich, but its action is not transitive.
\end{exm}

\begin{exm}{PN-homo-bad}
Although ultrahomogeneity has a quite natural definition (introduced in
Definition~\ref{dfn:ultrahomo}) which may lead to
a supposition that the automorphism group of an ultrahomogeneous measure is
rich, actually this property does not even imply that this group is nontrivial.
To convince oneself of that, observe that there exists a full non-atomic measure
\(\mu\) on the \Cantor{} such that on the algebra of all clopen subsets of
\(\Can\) it is both rational-valued and one-to-one. Then each partial \(\mu\)-isomorphism coincides with the identity map and hence \(\mu\) is ultrahomogeneous. Hence, \(\Homeo(\Can,\mu)\) is trivial, as the only homeomorphism preserving $\mu$ is the identity. This measure cannot be maximal, as then it would be good, so it would be preserved by a non-trivial homeomorphism.

(It is worth mentioning here that Akin
\cite{Akin99} constructed an uncountable family $\mathcal{M}^*$ of full non-atomic
measures on the Cantor space $\Can$ such that if for some continuous map $h\dd
\Can \to \Can$ and $\mu_1,\mu_2\in\mathcal{M}^*$ we have $h_*(\mu_1)=\mu_2$, then
$\mu_1 = \mu_2$ and $h$ is the identity map. This is in sharp contrast with
a classical theorem of Oxtoby and Ulam \cite{OU1941} stating that any two full
non-atomic measures on a Euclidean ball are homeomorphic.) 


General ultrahomogeneous measures are studied in detail in the forthcoming second part of
the present article \cite{DKMN2}. We focus there on those of them whose automorphism groups act transitively on the Cantor space.
\end{exm}

For the purpose of the next result, we recall classical notions and introduce
a few auxiliary ones. Let \(\OPN{Prob}(\Can)\) stand for the simplex of all
(probability) measures on the \Cantor{}, equipped with the standard \emph{weak$^*$}
topology (that is, \(\OPN{Prob}(\Can)\) is a metrisable compact convex set in
which measures \(\mu_n\) tend to \(\mu\) iff the integrals \(\int_{\Can} f
\dint{\mu_n}\) tend to \(\int_{\Can} f \dint{\mu}\) for any real-valued
continuous function on \(\Can\); equivalently, \(\lim_{n\to\infty} \mu_n = \mu\)
iff \(\lim_{n\to\infty} \mu_n(U) = \mu(U)\) for all clopen sets \(U \subseteq
\Can\).) For a fixed group-like set \(V\) denote by
\(\OPN{Full}(V)\) the set of all full measures whose clopen values set is
contained in \(V\). We consider \(\OPN{Full}(V)\) with the weak$^*$ topology (that
is, the topology induced from the weak$^*$ one of \(\OPN{Prob}(\Can)\)), and with \emph{strong} topology, that is, the topology of pointwise convergence on clopen sets w.r.t. the discrete topology of \(V\). More precisely, the strong topology of \(\OPN{Full}(V)\) is a metrisable topology in which measures \(\mu_n \in \OPN{Full}(V)\) tend to \(\mu \in \OPN{Full}(V)\) iff for any clopen set \(U\)
the sequence \(\mu_n(U)\) is eventually constant and equal to \(\mu(U)\). Note that \(\OPN{Full}(V)\) with the weak$^*$  topology is non-Polish in general. 

Below we consider \(\OPN{Full}(V)\) with the natural action of \(\Homeo(\Can)\).

\begin{pro}{PN-generic-measure}
Let \(V\) and \(\mu\) be, respectively, a group-like set and a good measure whose clopen values set coincides with \(V\).
Further, set $H$ to be the conjugacy class of $\mu$, that is, \(H = \{h^*(\mu)\dd h \in \Homeo(\Can)\}\). Then:
\begin{enumerate}[\upshape(a)]
\item The strong topology is Polish and stronger than the
 weak$^*$ topology on \(\OPN{Full}(V)\).
\item In the weak* topology, \(\OPN{Full}(V)\) is an \(\ffF_{\sigma\delta}\) subset of
 \(\OPN{Prob}(\Can)\).
\item In the weak$^*$ topology, \(H\) 
is a dense \(\ffF_{\sigma\delta}\)
 subset of \(\OPN{Full}(V)\).
\item In the strong topology, \(H\) is a unique dense \(\ggG_{\delta}\) conjugacy class in
 \(\OPN{Full}(V)\).
\end{enumerate}
\end{pro}
\begin{proof}
Part (a) is standard and left to the reader (recall that \(V\) is countable). To
show (b), observe that \(\OPN{Full}(V)\) coincides with \(\bigcap_U \bigcup_v
K(U,v)\) where \(U\) and \(v\) run over all, respectively, nonempty clopen
subsets of \(\Can\) and positive numbers from \(V\), and \(K(U,v)\) is
the (compact) set of all \(\nu \in \OPN{Prob}(\Can)\) satisfying \(\nu(U) = v\).
Now we pass to both (c) and (d). Note that a measure \(\lambda \in
\OPN{Full}(V)\) belongs to \(H\) iff \(\lambda\) is good and its clopen values
set is \(V\), or---equivalently---if for any nonempty clopen set \(U\) and each
\(v \in V\) that is less than \(\lambda(U)\) there exists a clopen set \(W\subseteq U\)
such that \(\lambda(W) = v\). In other words, \(H\) coincides with
\(\bigcap_{U,v} \bigcup_W E(U,v,W)\) where \(U\) and \(W\) run over all
nonempty clopen sets such that \(W \subseteq U\), and \(v\) runs over all
positive numbers from \(V\), and \(E(U,v,W)\) consists of all measures \(\nu \in
\OPN{Full}(V)\) for which \(\nu(U) \leq v\) or \(\nu(W) = v\). It is readily
seen that \(E(U,v,W)\) is closed in the weak$^*$ topology and open in the strong
one. So, it remains to show that \(H\) is dense in the strong topology (thanks
to (a); note also that uniqueness in part (d) follows from the Baire category
theorem). To this end, fix any measure \(\nu \in \OPN{Full}(V)\) and observe
that the sets of the form
\[B(\nu,P) = \{\rho \in \OPN{Full}(V)\dd\ \rho(p) = \nu(p)\ (\forall p \in P)\},\]
where \(P\) is a clopen partition of \(\Can\), form a basis of neighbourhoods of
\(\nu\) in the strong topology. But if \(P\) is a clopen partition and $\nu\in\OPN{Full}(V)$, then it
follows from both ultrahomogeneity and maximality of \(\mu\) that there exists
\(h \in \Homeo(\Can)\) such that \(h^*(\mu)\) extends the restriction of \(\nu\)
to \(P\), hence $h^*(\mu)\in B(\nu,P)$ and we are done.
\end{proof}

\section{Amalgamation and conjugacy classes}\label{sec:amalgam-1}

An \emph{equi-summed} matrix is a real-valued matrix $A=(a_{p,p'})_{p,p' \in P}$ with nonnegative entries, where $P$, denoted also by $P_A$ , is a finite set, such that
\begin{equation}\label{cond:equi-sum}
\sum_{p' \in P} a_{p,p'}= \sum_{p' \in P} a_{p',p} \qquad \text{ for all }p \in P.    
\end{equation}


An equi-summed matrix is  \emph{balanced} if it satisfies 
\begin{equation}\label{eq:normalised} 
\sum_{p,p' \in P} a_{p,p'}=1.
\end{equation}

We routinely identify $A=(a_{p,p'})_{p,p' \in P}$ with the corresponding labelled directed graph $G(A)$. That is, $G(A)$ is the graph whose vertex set is the subset of $P_A$ corresponding to non-zero rows of $A$ and labelled directed edges of $G(A)$ correspond to non-zero entries of $A$ meaning that $a_{p,p'}>0$ if and only if there is an edge in $G(A)$ from $p$ to $p'$ labelled by $a_{p,p'}$ and denoted $p\stackrel{a_{p,p'}}{\rightarrow} p' $ (or simply $p\to p'$ if the label is irrelevant).
Note that \eqref{cond:equi-sum} says that $A$ is equi-summed if for every vertex $p$ of $G(A)$ the sum of labels of all edges terminating at $p$ is equal to the sum of labels of edges outgoing from $p$. 

Let $A=(a_{p,p'})_{p,p' \in P}$, $B=(b_{q,q'})_{q,q' 
\in Q}$ be balanced matrices. A \emph{balanced morphism} $f$ from $B$ to $A$ is a surjection $f\colon Q \to P$ such that 
for every $p,p' \in P$ we have
\[ a_{p,p'}=\sum_{\substack{q \in f^{-1}(p),\\ q' \in f^{-1}(p')}} b_{q,q'}.\]

A \emph{cycle matrix} is an equi-summed matrix $C$ such that $G(C)$ is a single directed cycle. Since $C$ is equi-summed, all non-zero entries in $C$ are equal, and we call this number the \emph{weight} of $C$. The following seems to be a folklore fact:

\begin{lemma}\label{lem:MM-cycle}
Every equi-summed matrix $A$ is a sum of cycle matrices, whose weights belong to the additive subgroup of $\mathbb{R}$ generated by the entries of $A$.
\end{lemma}

\begin{proof}
The proof goes by induction on the number $m$ of non-zero entries of $A$. For $m=1$, this is obvious. Suppose the lemma is true for some $m$, and let $A$ be an equi-summed matrix with $m+1$ non-zero entries. Observe that \eqref{cond:equi-sum} ensures that there there is a cycle in $G(A)$, i.e., there are $k\ge 0$ and $i_0,i_1,\ldots,i_k=i_0\in P_A$ such that $i_m\neq i_n$ for $0\le m<n<k$ and $$a_{i_0,i_1}, a_{i_1,i_2}, \ldots ,a_{i_{k-1},i_k} = a_{i_{k-1},i_0}>0.$$ Let $a=\min \{a_{i_j,i_{j+1}}: 0\le j < k \}$, and let $C$ be the cycle matrix  corresponding to the cycle 
\[i_0\to i_1\to i_2\to \ldots \to i_{k-1}\to i_k = i_0 \]
with weight $a$. Then $A'=A-C$ has at most $m$ non-zero entries, and it is equi-summed. Clearly, all the entries in $A'$ are in the subgroup generated by the entries of $A$.  By the inductive assumption, $A'$ is a sum of cycle matrices with weights in the subgroup generated by the entries of $A$.  
\end{proof}

Let $V$ be a group-like set, let $((P_n,\mu_n),\pi_n)$ be a \fra chain from $\FF_V$, and let $\mu_V$ be the good measure on $\Can$ determined by $((P_n,\mu_n),\pi_n)$. Let $\Mm_V$ be the category of all $V$-valued balanced matrices $A=(a_{p,p'})_{p,p' \in P}$ such that $P$ is a clopen partition of $\Can$ and $A$ is \emph{compatible with $\mu_V$} meaning that
\begin{equation}\label{eqn:A-compatible+}
\sum_{p' \in P} a_{p,p'}=\mu_V(p) \qquad (\forall p \in P).
\end{equation}
If $P$ and $Q$ are different partitions of $\Can$, then two matrices $A=(a_{p,p'})_{p,p' \in P}$ and $B=(b_{q,q'})_{q,q' \in Q}$ in $\Mm_V$ are considered different, although it may happen that there is a bijection $\psi\colon P\to Q$ such that 
\[
a_{p,p'}=b_{\psi(p),\psi(p')} \qquad (\forall p,p' \in P).
\]
Since we consider only clopen partitions with nonempty cells, all rows of $A=(a_{p,p'})_{p,p' \in P}\in\Mm_V$ are non-zero. 

Morphisms in $\Mm_V$ are balanced morphisms that are also $\mu_V$-morphisms between compatible partitions. In other words, a balanced morphism $f$ between matrices $B$ and $A$ in 
$\Mm_V$ is a morphism in $\Mm_V$ if the transformation between the sets of indices $P_B\to P_A$ that determines $f$ is given by an $\mu_V$-morphism 
(see Section~\ref{sec:good-ultra}). Observe that if $B\in\Mm_V$ 
 and $r\colon P_B \to P$, where $P \in \Pp$, is a $\mu_V$-morphism, then there is $A=(a_{p,p'})_{p,p' \in P}\in\Mm_V$ such that $A$ is compatible with $P$ and and $r$ becomes an $\Mm_V$-morphism that maps $B$ onto $A$.

Let $\Cc_V \subseteq \Mm_V$ be the subcategory consisting of balanced matrices $A\in\Mm_V$ such that there is a single non-zero entry for each row of $A$.
For $A\in\Cc_V$ we easily see that $G(A)$ consists of disjoint cycles, that is, one can partition vertices of $G(A)$ into disjoint sets $\{S_1,\ldots,S_d\}$ so that $G(A)$ restricted to $S_j$ is a single directed cycle (of size $|S_j|$). 

Let $A=(a_{p,p'})_{p,p' \in P}$ be a balanced $V$-valued square matrix with non-zero rows. Since $\mu_V$ is maximal, using \eqref{eq:normalised} we easily find a partition $Q\in \Pp$ such that the elements of $Q$ are in a one-to-one correspondence with indices of $P$ of $A$ and \eqref{eqn:A-compatible+} holds.


Therefore every $V$-valued balanced matrix $A=(a_{p,p'})_{p,p' \in P}$ with non-zero rows is $\mu_V$ compatible with some clopen partition of $\Can$. We adopt a convention that for $A\in\Mm_V$ we write $P_A$ to denote a $\mu_V$-compatible partition of $\Can$ that is the set of indices of $A$. 



Because $\mu_V$ is maximal, Lemma \ref{lem:MM-cycle} immediately gives: 

\begin{corollary}\label{cor:MM-cofinal}
Let $V$ be group-like. Then $\Cc_V$ is cofinal in $\Mm_V$.
\end{corollary}

Given $R,Q\in\Pp$ and $\mu_V$-morphism $p\colon R \rightarrow Q$, we set
\[ [p]=\{ f \in \Homeo(\Can,\mu_V): f[r] \subseteq p(r) \,(\forall r \in R) \}. \]
In other words, $[p]$ consists of those $f \in  \Homeo(\Can,\mu_V)$ that send each cell $r$ of the partition $R$ into the cell $p(r)$ of the partition $Q$.

The collection $[p]$ where $p$ runs over all $\mu_V$-morphisms is a basis for the compact-open topology on $\Homeo(\Can,\mu_V)$. It is a standard fact in the theory of \fra \ limits (see \cite{Kub22}), that every morphism extends to an automorphism of the limit, i.e., $[p] \neq \emptyset$ for every $\mu_V$-morphism $p$.\par
We call \(\varphi \in \Homeo(\Can,\mu_V)\)
\emph{compatible} with \(A = (a_{p,p'})_{p,p' \in P} \in \Mm_V\) if \(\mu_V(\varphi[p] \cap p') = a_{p,p'}\) for all \(p, p' \in P\).
For every $A\in \Mm_V$ we set
\[
[A]=\{\varphi\in \Homeo(\Can,\mu_V):\varphi\text{ is compatible with $A$}\}.
\]
We distinguish a special class among all $\Mm_V$-morphism.  
We say that $B$ is a \emph{$\pi$-lift} of $A$ if it is a $\pi^{P_B}_{P_A}$-lift of $A$, that is, $B$ is a lift of $A$ via $\Mm_V$ morphism determined by $\pi^{P_B}_{P_A}$. Clearly, $[A] \supseteq [B]$ if $B \in \Mm_V$ is a $\pi$-lift of $A$.

\begin{proposition}
\label{pr:Basis}
The family $\{[A]: A \in \Mm_V\}$ is a basis for the compact-open topology on $\Homeo(\Can,\mu_V)$. Moreover, for every $A \in \Mm_V$, the family 
$$\{[B]: B  \text{ is a $\pi$-lift of } A \}$$
is a basis for $[A]$.
\end{proposition}

\begin{proof}
Clearly, for every $A \in \Mm_V$ the set $[A]$ is clopen in $\Homeo(\Can,\mu_V)$. Fix $f \in \Homeo(\Can,\mu_V)$. Let $p\colon R \rightarrow P$ be such that $f \in [p]$. Choose a common refinement \(Q \in \Pp\) of \(R\) and \(P\). Define $A=(a_{q,q'})_{q,q'\in Q}\in \Mm_V$ by $a_{q,q'} = \mu_V(f[q] \cap q')$ where $q, q' \in Q$. It is clear that $f\in [A]$ and $[A]\subseteq [p]$ (actually only the pairs \((q,q')\) of indices for which \(a_{q,q'}\) is non-zero matter). The other statement is proved in the same way.
\end{proof}

\begin{lemma}
\label{le:GoingUp}
Let $A \in \Cc_V$, and, for some $R \in \Pp$, let $p\colon R \rightarrow P_A$ be a $\mu_V$-morphism. There is $B \in \Mm_V$ such that $P_B=R$ and $p$ is an $\Mm_V$ morphism. 
from $B$ onto $A$.
\end{lemma}

\begin{proof} It is enough to consider each cycle separately. 
Suppose $x_0\to \ldots\to x_m \in P_A$ is a cycle in $G(A)$ (i.e., $x_0=x_m)$, where $m \ge 1$ with weight $1/m$. 



Fix $0\le k<m$ and set $Y^0=p^{-1}(x_k)$ and $Y^1=p^{-1}(x_{k+1})$ (recall that $x_0=x_m$). 
Take $i=0,1$. For each $y \in Y^i\subseteq R$ define $x^i_y=\mu_V(y)\in V$. 
Note that
\[
\sum_{y\in Y^i}x^i_y=\frac1m.
\]
Use Lemma \ref{le:CommonPartition} to find the set $\{z_1, \ldots, z_n\}\subseteq V$ and two partition $X^0$ and $X^1$ of $\{1, \ldots, n\}$ such that for $i=0,1$ and every $y\in Y^i$ there is the unique cell $Z^i_y\in X^i$ such that
\[
x^i_y=\mu_V(y)=\sum_{j\in Z^i_y}z_j.
\]
For every $y \in Y^0$ and $y' \in Y^1$ we set
\[
b_{y,y'}=\sum_{ j \in Z^0_y\cap Z^1_{y'}} z_j.
\]
Considering all $0\le k <m$ and setting all undefined entries as 0, we get the required $B=(b_{y,y'})_{y,y'\in R}\in\Mm_V$. 
\end{proof}

\begin{lemma}
\label{le:ReverseProjection}
Let 
$p\colon B\to A$ be an $\Mm_V$-morphism. There 
are $C\in\Mm_V$ and $\Mm_V$-morphism $r\colon C\to B$ such that $p \circ r=\pi^{P_C}_{P_A}$.

\end{lemma}

\begin{proof}
By Corollary \ref{cor:MM-cofinal}, we can assume that $B \in \Cc_V$. Since $((P_n,\mu_n), \pi_n)$ is a \fra sequence, there are $n$ and a $\mu_V$-morphism $r\colon P_n \rightarrow P_B$ such that $p \circ r=\pi^{P_n}_{P_A}$. By Lemma \ref{le:GoingUp}, there is $C \in \Mm_V$ such that $P_C=P_n$ and $r$ projects $C$ on $B$. 
\end{proof}

We say that a chain $(A_n,\pi_n)$ from $\Mm_V$ \emph{determines} $f \in \Homeo(\Can,\mu_V)$ if $\bigcap_n [A_n]=\{f\}$.

\begin{lemma}
\label{le:IsoConjugate}
Let 
$p\colon B\to A$ be an $\Mm_V$ morphism. Then $gfg^{-1} \in [A]$ for any $f \in [B]$ and $g \in [p]$.
\end{lemma}

\begin{proof}

Let $A=(a_{\alpha, \alpha'})_{\alpha, \alpha'\in P_A}$ and $B=(b_{\beta,\beta'})_{\beta,\beta'\in P_B}$.
Since $g[\beta] \subseteq p(\beta) \in P_A$ for each \(\beta\in P_B\), we get
\begin{equation}\label{eqn:PN-aux1000}
g^{-1}(\alpha) = \bigcup \{\beta\in P_B : \ \beta \in p^{-1}(\{\alpha\})\} \qquad
(\forall \alpha \in P_B).
\end{equation}
Fix $\alpha,\alpha'\in P_A$.  We have
\begin{multline}\label{eq:aux-DK-0001}
a_{\alpha,\alpha'} = 
\sum_{\substack{\beta \in p^{-1}(\{\alpha\})\\\beta' \in p^{-1}(\{\alpha'\})}}
b_{\beta,\beta'} = 
\sum_{\substack{\beta \in p^{-1}(\{\alpha\})\\\beta' \in p^{-1}(\{\alpha'\})}}
\mu_V(f[\beta] \cap \beta')\\
= \mu_V\Bigl(f\bigl[\bigcup \{\beta\in P_B:  p(\beta) = \alpha\}\bigr] \cap
\bigl(\bigcup \{\beta'\in P_B: \ p(\beta') = \alpha'\}\bigr)\Bigr)\\
= \mu_V(f[g^{-1}[\alpha]] \cap g^{-1}[\alpha']) = \mu_V(gfg^{-1}[\alpha] \cap
\alpha'),
\end{multline}
where the first equality in \eqref{eq:aux-DK-0001} follows from the fact that $p$ is
an \(\Mm_V\)-morphism, the second follows from \(f \in [B]\), 
 third uses that $P_B$ is a partition, 
 we get the fourth equality from \eqref{eqn:PN-aux1000},  
 and the last equality follows from the invariance of \(\mu_V\) under \(g\). Clearly, \eqref{eq:aux-DK-0001} means that \(gfg^{-1} \in [A]\).
\end{proof}

\begin{corollary}\label{cor:MM-compatible-nonempty}
For every $A \in \Mm_V$ we have $[A] \neq \emptyset$.
\end{corollary}

\begin{proof}
By Corollary \ref{cor:MM-cofinal}, there is $C \in \Cc_V$ and an $\Mm_V$-morphism $p$ that projects $C$ onto $A$. Moreover, since $C\in \Cc_V$, we see that $C$ determines a permutation of $P_C$, hence it determines a partial $\mu_V$-isomorphism $f_C\colon P_C \to P_C$. By ultrahomogeneity of $\mu_V$, there is $\varphi \in \Homeo(\Can,\mu_V)$ such that $\varphi[x]=f(x)$ for $x \in P_C$, i.e., $\varphi \in [C]$. By Lemma \ref{le:IsoConjugate}, $[A] \neq \emptyset$.
\end{proof}
Now we have all the ingredients to connect directedness, resp. weak amalgamation of $\Mm_V$ with the Rokhlin property, resp. strong Rokhlin property. We remark that such characterizations are essentially known to experts and were explicitly stated in different contexts in \cite{KeRo07} and \cite{Iv99}. However, since there is no direct reference applicable in our context we provide full proofs below.
\begin{theorem}
	\label{th:Dense} Assume that $V$ is group-like. 
	The following are equivalent:
	
	\begin{enumerate}[\upshape(1)] 
        \item \label{c:i} $\Mm_V$ is directed,
		\item \label{c:ii} $\Homeo(\Can,\mu_V)$ has the Rokhlin property.
	\end{enumerate}
\end{theorem}

\begin{proof}
\eqref{c:i}$\Rightarrow$\eqref{c:ii} 
Fix $A, B \in \Mm_V$. 
Since $\Mm_V$ is directed, we find $C \in \Mm_V$ such that there are $\Mm_V$-morphisms $p\colon C \to A$ and $q\colon C \to B$. By Lemma \ref{le:IsoConjugate}, there is an open set $U\subseteq [B]$ such that for every $f \in U$ we can pick  $\varphi \in \Homeo(\Can,\mu_V)$ with $\varphi f\varphi^{-1} \in [A]$. As $B$ was arbitrary, it follows that the set of all $f \in \Homeo(\Can,\mu_V)$ such that there is $\varphi \in \Homeo(\Can,\mu_V)$ with $\varphi f \varphi^{-1} \in [A]$ is open and dense in $\Homeo(\Can,\mu_V)$. As $A$ was arbitrary, the set of $f \in \Homeo(\Can,\mu_V)$ with a dense conjugacy class is comeager; in particular, it is not empty. 
	
\eqref{c:ii}$\Rightarrow$\eqref{c:i} Fix $f \in \Homeo(\Can,\mu_V)$ with a dense conjugacy class and a chain $(B_i)$ that determines $f$. Then $\{B_i\}$ witnesses that $\Mm_V$ is directed.
\end{proof}

\begin{theorem}
\label{th:CAMImpliesNonMeager}
If $V$ is group-like and $\Mm_V$ has weak amalgamation, then there is a non-meager conjugacy class in $\Homeo(\Can,\mu_V)$.
\end{theorem}

\begin{proof}
Fix a neighbourhood $U_0$ of 
$\text{id}_{\Can}\in  \Homeo(\Can,\mu_V)$. Let
$U\subset U_0$ be such that $\text{id}_{\Can}\in U$, $U=U^{-1}$ and $UU\subseteq U_0$. Let $A$ be the identity matrix with the indexing partition $P_A\in \Pp$ that is sufficiently fine to give $[A]\subseteq U$. We will verify that Rosendal's criterion \cite[Proposition 3.2]{BeMeTs} holds. 

Fix $B \in \Mm^A_V$ that is a weak amalgamation basis for $A$. By Lemma \ref{le:ReverseProjection}, we can assume that $\pi^{P_B}_{P_A}$ is an $\Mm_V$-morphism witnessing it. In particular, $[B] \subseteq [A]$.

By Proposition \ref{pr:Basis}, $\{[B']: B' \mbox{ is a } \pi\mbox{-lift of } B  \}$ is a basis for $[B]$. Fix $\pi$-lifts $C,D$ of $B$. Use weak amalgamation to find $E \in \Mm^B_V$  with morphisms $p\colon E \to C$ and $q\colon E \to D$ such that 
$$\pi^{P_C}_{P_A} \circ p=\pi^{P_B}_{P_A}\circ  \pi^{P_C}_{P_B} \circ p= \pi^{P_B}_{P_A}\circ  \pi^{P_D}_{P_B} \circ q =\pi^{P_D}_{P_A} \circ q.$$
By Lemma \ref{le:ReverseProjection}, there is $F \in \Mm_V$ and am $\Mm_V$-morphism $r\colon F \to E$ such that
$$\pi^{P_D}_{P_A} \circ q \circ r=\pi^{P_F}_{P_A}.$$
%
But this means that if $g \in [q \circ r]$, then $g \in [A]$. Similarly,  $\pi^{P_C}_{P_A} \circ p= \pi^{P_D}_{P_A} \circ q$, gives $[p \circ r]\subseteq [A]$. Now, exactly as in the proof of Theorem \ref{th:Dense}, we argue that there is $f \in [B]$ such that $\{gfg^{-1}: g \in [A]\}$ is dense in $[B]$. As $U_0$ was arbitrary, by Rosendal's criterion \cite[Proposition 3.2]{BeMeTs}, we get that there is a non-meager conjugacy class in $\Homeo(\Can,\mu_V)$.
\end{proof}

\begin{corollary}
\label{co:Directed&CAMImpliesNonMeager}
If $\Mm_V$ is directed and has weak amalgamation, then $\Homeo(\Can,\mu_V)$
has the strong Rokhlin property.
\end{corollary}

\section{Amalgamation in $\Mm_V$}\label{sec:amalgam-2}
Let $V$ be a group-like set and take $v\in V$. Let $\CC^v_V$ be the set of equi-summed matrices $A$ with entries in $V$ summing up to $v$ and such that there is exactly one non-zero entry in each row of $A$. Each matrix in $\Cc^v_V$ can be uniquely  represented as a tuple $\vec u =((v_1,n_1),\ldots,(v_m,n_m))$, where $v_i\in V$ and $n_i\in\NN$ for $1\leq i\leq m$ satisfy $\sum_{i=1}^m n_i v_i=v$. Up to an indexing partition of $\Can$, we can identify $\Cc_V$ with $\Cc_V^1$.


Let $v,w\in V$ be such that $v+w\leq 1$. If $\vec c=((v_1,n_1),\ldots,(v_m,n_m))\in\CC_V^v$ and $\vec d=((w_1,k_1),\ldots,(w_l,k_l))\in\CC_V^w$ are two tuples, then 
their \emph{sum} $\vec c+\vec d$ in an element of $\CC_V^{v+w}$ which is simply the concatenation of $\vec c$ and $\vec d$. Similarly, if $n\in\NN$ is such that $nv\in V$ and $\vec c=((v_1,n_1),\ldots,(v_m,n_m))\in\CC_V^v$, then $n\cdot \vec c=\vec c+\ldots+\vec c$ ($n$ times), hence $n\cdot \vec c\in \CC_V^{nv}$. If $v=w$, then by a \emph{morphism} from $\vec c$ 
 onto $\vec d$ 
we understand a partition of $\{1,\ldots,m\}$ into disjoint subsets $A_1,\ldots,A_l$ such that for every $1\leq j\leq l$ we have $\sum_{i\in A_j} n_i v_i=k_j w_j$, and $\frac{n_i}{k_j}\in\NN$ for every $i\in A_j$.

Finally, if $\varphi\colon \vec c_0\to \vec c_1$ and $\psi\colon \vec d_0\to \vec d_1$ are morphisms where $\vec c_0, \vec c_1\in \CC_V^v$ and $\vec d_0,\vec d_1\in \CC_V^w$, then $\phi+\psi$ is the obvious morphism between $\vec c_0+ \vec d_0$ and $\vec c_1+\vec d_1$. 

\begin{proposition}\label{prop:cofinalcharacterization}
Let $V$ be a group-like set. Then $\CC_V$ has cofinal amalgamation if and only if for every $v\in V$ there are $v_1,\ldots,v_n\in V$ such that $v=v_1+\ldots +v_n$ and for every $1\leq i\leq n$, the set $\CC_V^{v_i}$ is directed.
\end{proposition}
\begin{proof}
$\Rightarrow$ Fix $v\in V$. Since $\CC_V$ has cofinal amalgamation there exists a tuple $((w_1,n_1),\ldots,(w_k,n_k))\in \CC_V^v$ witnessing the cofinal amalgamation property. We claim that $\CC_V^{w_i}$ is directed for $1\leq i\leq k$. Then we are done, since $v=\sum_{i=1}^k n_i w_i$. Let $\vec t_i,\vec u_i\in\CC_V^{w_i}$ be arbitrary. We need to find $\vec s_i\in\CC_V^{w_i}$ that gets mapped on $\vec t_i$ and $\vec u_i$. Notice that the tuples $(w_1,n_1)+\ldots+n_i\vec t_i+\ldots+(w_k,n_k)$ and $(w_1,n_1)+\ldots+n_i\vec u_i+\ldots+(w_k,n_k)$ are mapped onto $((w_1,n_1),\ldots,(w_k,n_k))$. By assumption, there is their amalgmam $((z_1,m_1),\ldots,(z_l,m_l))$. Without loss of generality, we may assume there is $1\le j\leq l$ such that $((z_1,m_1),\ldots,(z_j,m_j))$ gets mapped onto $n_i\vec t_i$ and $n_i\vec u_i$. Then clearly $((z_1,m_1/n_i),\ldots,(z_j,m_j/n_i))$ maps onto $\vec t_i$ and $\vec u_i$,  which proves directedness. 

$\Leftarrow$ Fix $((v_1,n_1),\ldots,(v_k,n_k))\in\CC_V$. It is enough to prove that $(v_i,n_i)$ witnesses cofinal amalgamation in $\CC_V^{n_iv_i}$ for $1\leq i\leq k$, since the sum of amalgams for $i=1,\ldots,k$ is an amalgam of the sum. Fix $1\leq i\leq k$. Without loss of generality, $n_i=1$. By the assumption there are $w_1,\ldots,w_n\in V$ such that $v_i=w_1+\ldots+w_n$ and $\CC_V^{w_j}$ is directed for each $1\leq j\leq n$. 
Repeating the arguments from the other implication, it is straightforward to see  that the tuple $((w_1,1),\ldots,(w_n,1))$ gets mapped onto $(v_i,n_i)$, so it  witnesses the cofinal amalgamation.
\end{proof}
Ibarlucía and Melleray characterized good measures $\mu$ for which $\Homeo(\Can,\mu)$ contains a dense conjugacy class (has the Rokhlin property), see \cite[Proposition 6.6]{IM16}. For the reader's convenience, we provide our concise proof. 
\begin{lemma}\label{lem:ringlike-densityproperty}
If $\mu$ is a good measure with ring-like 
$V_\mu$, then
$\Homeo(\Can,\mu)$ has the Rokhlin property.
\end{lemma}
\begin{proof}
By Theorem \ref{th:Dense}, it suffices to show that $\CC_V$ is directed. Pick $\vec c =((v_i,n_i))_{i\in I}$, $\vec d =((w_j,m_j))_{j\in J}$ 
in $\CC_V$. As $V$ is ring-like, $v_iw_j\in V$, for each $(i,j)\in I\times J$.
Hence, $\vec u=\big((v_iw_j,n_im_j\big)_{(i,j)\in I\times J}\in \CC_V$.
It is easy to see that $\vec u$ projects on $\vec c$ and $\vec d$.
\end{proof}
\begin{remark}
 Alternatively, by \cite[Proposition 3.9]{AkDoMauYi08}, if $V$ is ring-like then $\mu$ is a push-forward of $\mu^\NN$ via a homeomorphism $\psi\colon (\Can)^\NN\to\Can$. Hence, $\Homeo((\Can)^\NN,\mu^\NN)$ is topologically isomorphic to $\Homeo(\Can,\mu)$. Clearly, $\psi\Phi\psi^{-1}$, where $\Phi=\prod_{n\in\NN} \varphi_n$ has a dense conjugacy class 
 in $\Homeo(\Can,\mu)$. 
\end{remark}

\begin{lemma}\label{lem:density-divisibility}
Let $\mu$ be a good measure 
such that $\Homeo(\Can,\mu)$ has the Rokhlin property. If $n\in\NN$ is such that $\frac1n\in V_\mu$, then for every $v\in V_\mu$ we have $\frac{v}{n} \in V_\mu$. In particular, $Q:=\{n \in \NN : \frac1n \in V_\mu\}$ is a subsemigroup of \((\NN,\cdot)\) closed under taking positive divisors, and \(Q^{-1} \cdot V \subset V\) where \(Q^{-1} = \{1/n : \ n \in Q\}\).
\end{lemma}
\begin{proof}
Note that $\CC_V$ is directed by Theorem~\ref{th:Dense}. Pick $v\in V$. Hence, there is $\vec c\in \CC_V$ projecting onto both, $(\frac1n,n)\in\CC_V$ and $((v,1),(1-v,1))\in\CC_V$. Let $(w_1,n_1),\ldots,(w_k,n_k)$ be those cycles in $\vec c$ that are mapped onto $(v,1)$. Since they also get mapped onto $(\frac1n,n)$, we infer that $n$ divides $n_i$ for $1\leq i\leq k$. Then it is easy to check that $\frac{v}{n} = \sum_{i=1}^k \frac{n_i}{n} w_i\in V$. 

Now if \(n,m \in Q\), then the above argument applies to \(v = 1/m\). Consequently, \(\frac{1}{nm} \in V\) and thus \(n \cdot m \in Q\). The last claim of the lemma follows from the first part of the proof.
\end{proof}
\begin{corollary}
Let $V$ be group-like. If $\Homeo(\Can,\mu_V)$ has the Rokhlin property, then either $V \cap \QQ=\{0,1\}$ or $V$ is module-like, i.e., $V=M \cap [0,1]$, where $M$ is a module over a dense subring $R \subseteq \QQ$.
\end{corollary}

\begin{corollary}
If \(V\) is a group-like set that contains \(\QQ \cap [0,1]\) but is not \(\QQ\)-like, then \(\Homeo(\Can,\mu_V)\) does not have the Rokhlin property.
\end{corollary}

\noindent Next result characterizes 
the existence of dense/comeager conjugacy classes in $\Homeo(\Can,\mu)$ for good measures $\mu$ with $V_\mu\subseteq\QQ$. Let $\bP$ be the set of primes.
\begin{theorem}\label{thm:Q-ring}
Let $\mu$ be a good measure on $\Can$ with $V_\mu\subseteq \QQ$. 
\TFCAE
\begin{enumerate}[\upshape(1)]
    \item\label{cond:Q:i} $\Homeo(\Can,\mu)$ has the strong Rokhlin property.
    \item\label{cond:Q:ii} $\Homeo(\Can,\mu)$ has the Rokhlin property.
    \item\label{cond:Q:iii} $V_\mu$ is ring-like.
\end{enumerate}
\end{theorem}
\begin{proof}
If $V$ is ring-like, then $\CC_V$ is directed by Theorem~\ref{th:Dense} and Lemma~\ref{lem:ringlike-densityproperty}. By Corollary \ref{co:Directed&CAMImpliesNonMeager}, it suffices to show that it also has cofinal amalgamation. We apply Proposition~\ref{prop:cofinalcharacterization}. Let $v\in V$. Write $v=r/q$, where $r$ and $q:=p_1^{n_1}\cdots p_k^{n_k}$ are coprime, $p_1,\ldots,p_k$ are distinct primes, and $n_1,\ldots,n_k\in\NN$. It is enough to show that $\CC_V^{1/q}$ 
is directed. However, $\CC_V^{1/q}$ 
is directed if and only if $\CC_{V'}$ is directed, where $V'=\{
qw : w\in V\cap [0,1/q ]\}$.
Clearly, $V'$ is ring-like, so we conclude using Theorem~\ref{th:Dense} and Lemma~\ref{lem:ringlike-densityproperty}.

Now, suppose $V$ is not ring-like. Then $V$ is uniquely determined by the set $J:=\{p^{-n}\in V :p\in\bP,\ n\in\NN\}$.
Moreover, $V$ is ring-like if and only if there is
$\emptyset\neq P\subseteq\bP$ such that if $p\in\bP$ and $n\in\NN$ then $p^{-n}\in V$ if and only if $p\in P$. 
If $\Homeo(\Can,\mu)$ had the Rokhlin property, then using Lemma~\ref{lem:density-divisibility} we could find $P\subseteq\bP$ as above, which would contradict the fact that $V$ is not ring-like.
\end{proof}

For $W \subseteq \QQ$ and $p\in\bP$, let $n^W_p \in \NN \cup \{0,\infty\}$ be equal $\sup\{n\in\NN \cup \{0\}\dd\ 1/p^n\in W\}$. 
By \cite[Sec.~85, Chap.~XIII]{Fu}, if $W$ is a group that contains \(\mathbb{Z}\), then $W$ is characterized by the sequence $(n^V_p)_{p\in\bP}$. 

\begin{corollary}
Let $\mu$ be a good measure 
with $V_\mu \subseteq \QQ$. 
Then 
$\Homeo(\Can,\mu)$ has the strong Rokhlin property 
if and only if $n^{V_\mu+\ZZZ}_p \in \{0,\infty\}$ for every $p\in\bP$.   
\end{corollary}

Finally, we provide a few results about good measures $\mu$ such that $V_\mu\not\subseteq \QQ$. 
We start with a simple characterization of $\QQ$-like sets.

\begin{lemma}\label{lem:Q-in-V}
Assume $V$ is group-like. Then $V$ is $\QQ$-like if and only if $\QQ\cap[0,1]\subseteq V$ and \(\Homeo(\Can,\mu_V)\) has the Rokhlin property.
 \end{lemma}
 \begin{proof}
The `only if' direction follows from Akin's result \cite{Akin05} (cf. Corollary~\ref{co:GoodAmalg} below). The other direction is an immediate consequence of Lemma~\ref{lem:density-divisibility}.
 \end{proof}

Let $A,B \in \Cc_V$ and $p\colon B \rightarrow A$ be a morphism. The \emph{$p$-winding number} of a cycle $C$ in $B$ is the cardinality of $p^{-1}(p(c))$ for $c \in C$. The $p$-winding number of $C$ is the number of times $C$ winds over the cycle it is mapped onto by $p$, so it is well-defined. 

\begin{theorem}
\label{th:Amalg}
Assume that $V$ is $\QQ$-like. Then $\Cc_V$ has amalgamation.
\end{theorem}
\begin{proof}
Fix $A,B_0,B_1 \in \Cc_V$. Take morphisms $p_0\colon B_0 \rightarrow A$ and  $p_1\colon B_1 \rightarrow A$. Clearly, we can assume that $A$ is a single cycle of size $n$. Because $V$ is $\QQ$-like, without loss of generality, we can assume that for $i=0,1$ each cycle in $B_i$ has the $p_i$-winding number $1$, 
i.e., all the cycles in $B_i$ 
have size $n$. 

Let $x^0_i \in V$ ($1\leq i \leq k$) and $x^1_j \in V$ ($1\leq j \leq l$) be weights of the cycles in $B_0$, respectively, $B_1$. Let $z$ be the weight of the unique cycle in $A$. Then Lemma \ref{le:CommonPartition} provides weights $z_s \in V$ ($1\leq s \leq m$) and partitions $X^i_0$ ($1\leq i \leq k$), $X^j_1$ ($1\leq j \leq l$) of $\{1,\ldots,m\}$. We set $C$ to consist of  $m$ cycles of size $n$ indexed  as $c^s_1,\ldots,c^s_n$ with weight $z_s$ ($1\leq s\leq m$). To define morphisms $q_i\colon C\to B_i$ ($i=0,1$), 
 we first enumerate vertices of the unique cycle in $A$ as $a_1,\ldots,a_n$. Then, 
 for every $1\leq t\leq n$ and $1\leq s\leq m$
we find indices $1\leq i\leq k$ and $1\leq j\leq l$ such that $s\in X_0^i\cap X_1^j$, then we find cycles $B'_0$ in $B_0$ and $B'_1$ in $B_1$ with weights $x^i_0$, respectively $x^j_1$, and finally, we set $q_0(c_t)=b^0_t$ and  $q_1(c_t)=b^1_t$, 
where $b^r_t$ ($r=0,1$) is the unique vertex of the cycle in $B'_r$ such that $p_r(b^r_t)=a_t$.
It is straightforward to verify that $q_0$ and $q_1$ are morphisms witnessing that $C$ amalgamates $B_0$ and $B_1$ over $A$.
\end{proof}

\begin{corollary}[{\upshape\cite[Theorem 4.17]{Akin05},\cite[Proposition 6.8]{IM16}}]
\label{co:GoodAmalg}
Assume that $V$ is $\QQ$-like. Then $\Homeo(\Can,\mu_V)$ has the strong Rokhlin property.
\end{corollary}

Given $V$ and $a\in V\setminus\{0\}$, define $V_a:=\{v/a\colon v\in V\}\cap [0,1]$. Note that $V_a$ is also group-like.
\begin{proposition}\label{pro:dychotomy-dense}
Let $V$ be a group-like set. For every $a\in V\setminus\{0\}$, denote by $\mu_a$ the good measure whose clopen values set is $V_a$. Then the following dichotomy holds.
\begin{enumerate}[\upshape(1)]
    \item If $V$ is $\QQ$-like then $\Homeo(\Can,\mu_a)$ has the strong Rokhlin property for every positive $a\in V$.
    \item If $V$ is not $\QQ$-like then for $a\in V$ arbitrarily close to
 $1$, $\Homeo(\Can,\mu_a)$ does not have the Rokhlin property. More specifically, if $b\in V$ and $n>1$ are such that $b<1$ and
 $\frac{b}{n} \notin V$, then $\Homeo(\Can,\mu_{na})$ does not have a dense conjugacy class for each $a \in V \cap [b/n,1/n]$.
\end{enumerate}
\end{proposition}
\begin{proof}
If $V$ is $\QQ$-like then the result follows from Corollary~\ref{co:GoodAmalg}.

If $V$ is not $\QQ$-like there is $b \in V \cap (0,1)$
and $n > 1$ such that $\frac{b}{n} \notin V$. Take arbitrary $c \in V \cap
[b/n,1/n]$ (recall that $V$ is dense in $[0,1]$) and consider $a := nc
\in V$. Then we can find a clopen partition $\{d_1,\ldots,d_n\}$ of $\Can$ such
that $\mu_a(d_k) = \frac{1}{n}$ for each $k$ (as $\frac{1}{n} \in V_a$).
We can also find a clopen set $e$ with $\mu_a(e) = b/a$. If $\Homeo(\Can,\mu_a)$ had a dense conjugacy class, then by Lemma~\ref{lem:density-divisibility} we have $b/(na)\in V_a$, thus $b/n\in V$ which contradicts the assumption.
\end{proof}

\bibliographystyle{siam}
\bibliography{references}
\end{document}